\documentclass{article}
\usepackage[utf8]{inputenc}
\usepackage{mlmath}
\usepackage{array,multirow}
\usepackage{extarrows}
\usepackage{amsthm}
\usepackage{graphicx}
\usepackage{float}
\usepackage{subfigure}
\usepackage{listings}
\usepackage{color}
\usepackage{appendix}
\usepackage{grffile}
\usepackage{caption}
\usepackage{float}
\usepackage{blindtext}
\usepackage{titlesec}
\usepackage{bm}
\usepackage{bbm}
\usepackage{tikz}
\usepackage{authblk}
\usetikzlibrary{shapes.geometric, arrows}
\usetikzlibrary{positioning}

\tikzstyle{startstop} = [rectangle, rounded corners, 
minimum width=3cm, 
minimum height=1cm,
text centered, 
draw=black, 
fill=red!30]

\tikzstyle{io} = [trapezium, 
trapezium stretches=true, 
trapezium left angle=70, 
trapezium right angle=110, 
minimum width=1.5cm, 
minimum height=1cm, text centered, 
draw=black, fill=blue!30]

\tikzstyle{process} = [
rectangle, 
minimum width=3cm, 
minimum height=1cm, 
text centered, 
text width=3cm, 
draw=black, 
fill=orange!30]

\tikzstyle{decision} = [diamond, 
minimum width=3cm, 
minimum height=1cm, 
text centered, 
draw=black, 
fill=green!30]

\tikzstyle{arrow} = [thick,->,>=stealth]

\definecolor{dkgreen}{rgb}{0,0.6,0}
\definecolor{gray}{rgb}{0.5,0.5,0.5}
\definecolor{mauve}{rgb}{0.58,0,0.82}

\theoremstyle{plain}
\newtheorem{thm}{Theorem}[section]
\newtheorem{lem}[thm]{Lemma}

\newtheorem{cor}{Corollary}[thm]

\theoremstyle{definition}

\theoremstyle{remark}
\newtheorem*{rem}{Remark}

\usepackage[pdftex,linkcolor=black,citecolor=black,backref=page]{hyperref}
\hypersetup{
colorlinks=true,
linkcolor=black,
anchorcolor=black,
citecolor=black}

\providecommand{\keywords}[1]
{
  \small	
  \textbf{\textit{Keywords---}} #1
}

\title{Blessing of High-Order Dimensionality: from Non-Convex to Convex Optimization for Sensor Network Localization}

\author[1]{Mingyu Lei\thanks{mr.ray@sjtu.edu.cn}}
\author[1]{Jiayu Zhang\thanks{jswd50x-1x1zshn@sjtu.edu.cn}}
\author[2]{Yinyu Ye\thanks{yyye@stanford.edu}}
\affil[1]{Shanghai Jiao Tong University}
\affil[2]{Standford University}
\date{}

\begin{document}
\maketitle
\begin{abstract}
This paper investigates the Sensor Network Localization (SNL) problem, which seeks to determine sensor locations based on known anchor locations and partially given anchors-sensors and sensors-sensors distances. 
Two primary methods for solving the SNL problem are analyzed: the low-dimensional method that directly minimizes a loss function, and the high-dimensional semi-definite relaxation (SDR) method that reformulates the SNL problem as an SDP (semi-definite programming) problem.
The paper primarily focuses on the intrinsic non-convexity of the loss function of the low-dimensional method, which is shown in our main theorem.
The SDR method, via second-order dimension augmentation, is discussed in the context of its ability to transform non-convex problems into convex ones; while the first-order direct dimension augmentation fails.
Additionally, we will show that more edges don't necessarily contribute to the better convexity of the loss function. 
Moreover, we provide an explanation for the success of the SDR+GD (gradient descent) method which uses the SDR solution as a warm-start of the minimization of the loss function by gradient descent.
The paper also explores the parallels among SNL, max-cut, and neural networks in terms of the blessing of high-order dimension augmentation.
\end{abstract}

\section{Introduction}\label{Introduction}

A sensor network is a network made up of devices equipped with rangefinders. 
The nodes of the network are bipartite, comprising anchors and sensors. 
The positions of the anchors are fixed and known. 
The sensor network localization (SNL) problem aims to determine sensor locations based on actual anchor locations and partially given anchors-sensors and sensors-sensors distances, taking into account possible noise. 

In this paper, we explore two important methods for solving the SNL problem.
The first method involves minimizing a certain loss function without constraints using an iterative optimization algorithm; see, e.g., \cite{biswas_lian_wang_ye_2006, biswas_five_people_2006, pong_tseng_2010, tseng_2007, wang2008}. 
The second method, known as the semi-definite relaxation (SDR) method, transforms the original non-convex optimization problem into a convex one by second-order dimension augmentation; see, e.g., \cite{biswas_five_people_2006, biswas_ye_revised, by2004, so2007, wang2008}.
We should emphasize here that the loss function has different forms.
The differences among different forms are the degrees of the terms.
One of the forms is a polynomial, where the inside degree and the outside degree are both $2$.
For each form of the loss function, we can augment its dimension.
The well-studied SDR method is the dimension augmentation of the form with inside degree $2$ and outside degree $1$.

While extensive research has been done on the algorithmic aspect of both methods, the geometric landscape of their optimization formulations remains largely unexplored, particularly for the first method. 
In this paper, we focus on the non-convex property of the loss function, which is important for the analysis of the first method.

The SDR method involves increasing the problem's dimension to transform it into a convex optimization problem.
This allows for the resolution of the original non-convex problem.
Notably, the SDR+GD (gradient descent) method, which employs the SDR solution as a warm-start for the loss function method solved by GD, exhibits superior performance in practice \cite{biswas_lian_wang_ye_2006, biswas_five_people_2006, wang2008}.

To simplify the discussion, we will refer to the first method as the ``loss function method" since our primary focus lies in analyzing the non-convex property of the loss function that is used by this method.
Additionally, we emphasize that the loss function method is low-dimensional compared to the SDR method, which is high-dimensional. 
The ``low-dimensional" here has two meanings.
The first is that the loss function method has fewer optimization variables than the SDR method.
For instance, the phrase ``dimension augmentation" in our paper means that one considers a new optimization problem that has more optimization variables compared with the original low-dimensional optimization problem. 
The second is the geometric meaning of the solutions provided by the two methods.
It is shown by \cite{so2007} that the solution to the SDR optimization problem, which has more variables, can be considered as a graph realization of the original SNL problem in a higher-dimensional space.

Additionally, the dimension augmentation employed by the SDR method is a product-type dimension augmentation, therefore a second-order dimension augmentation.
Here, the order of the dimension augmentation is determined by the highest number of original variables that are multiplied together to form the augmented variables.
In the works of Lasserre \cite{lasserre2007sum,lasserre2009moments}, the benefits of high-order dimension augmentation are outlined.
Lasserre proved that as the order of the dimension augmentation of the original general polynomial optimization increases, the solution to the dimension-augmented optimization problem will converge to the solution to the original polynomial optimization problem.
Here, the dimension augmentation of the original polynomial optimization is always an SDP.
Furthermore, under certain circumstances, finite convergence is guaranteed \cite{lasserre2007sum, lasserre2009moments, finite_convergence}.
The second-order dimension augmentation, i.e., the SDR method, of the loss function method has the same solution as the loss function method under certain circumstances as we will show in Section \ref{properties of the solution matrix}.
Therefore, this is an example of finite convergence, although our original optimization problem is slightly different from a polynomial optimization problem.
In the context of SNL, Nie proposed a fourth-order dimension augmentation of the polynomial form of the loss function method \cite{nie2009sum}.
This is also an SDP and the finite convergence is guaranteed under certain circumstances.
However, as mentioned above, finite convergence can be guaranteed using only second-order dimension augmentation;
that is, second-order is enough in the context of SNL.
Nonetheless, the direct dimension augmentation, or the first-order dimension augmentation of the loss function method, remains non-convex.
This will be proved in Section \ref{Non-Convexity of Direct Dimension Augmentation}.

Similar findings have also been observed in the field of neural networks.
Large neural networks increase the dimension of the problem and extract more information from the dataset; transfer learning can also be considered as generating a warm-start for new models, with probably less trainable parameters or a cheaper training process.
Both increasing the size of neural networks and transfer learning contribute to the good performance of the neural networks.
Therefore, we will give a comparison of the two scenarios for SNL and neural networks.

\noindent\textbf{Related works}

Biswas and Ye \cite{by2004} reformulated sensor network localization (SNL) as a semi-definite programming (SDP) problem. 
Their pioneering work involved the development of the semi-definite relaxation (SDR) technique for SNL, which transforms the SNL problem into a convex optimization problem and exhibits outstanding performance. 
However, the SDR solutions to SNL problems that are nearly ill-posed or affected by noise are sometimes inaccurate.
Furthermore, from an algorithmic point of view, solving the SDP given by the SDR method remains to be time-consuming today.
We refer the readers to \cite{biswas_ye_revised, tseng_2007, wang2008, pong_tseng_2010} for some methods to accelerate the standard SDR proposed in \cite{by2004}. 

Later, several methods were presented to improve the accuracy and speed of SDR in SNL problems affected by noise. 
For instance, heuristics such as reusing the SDR solution to regularize the problem or using a warm-start showed outstanding performance \cite{biswas_five_people_2006, biswas_lian_wang_ye_2006}.
Research such as \cite{zhu2010universal, so2007, sophd} gave good theoretical results on the original version of SDR \cite{by2004}. 
However, their results cannot fully explain why the regularized SDR works in noisy cases. 
It is also not clear when the minimization (for example, gradient descent) will converge to the global minimizer during the stage of minimizing the loss function.

The literature on source localization suggested that, in the one-sensor case of SNL, the loss function is non-convex in general \cite{source_2008, source_2018}.
Additionally, better performance of minimizing the loss function when a warm-start is used was reported in \cite{biswas_five_people_2006}.
However, no rigorous results on the commonly encountered SNL loss functions have been presented and the geometric landscape of the loss function remains largely unexplored.

Another line of work involves implementing multidimensional scaling (MDS) techniques, as seen in works such as \cite{mds_2004, mds_survey, mds_2003, mds_2006}. 
MDS methods can provide a rough estimation for further minimization of the loss function; see, e.g., \cite{galp2023}. 
However, benchmarks in \cite{galp2023} and figures in \cite{mds_2003, mds_2006} indicated that MDS accuracy is lower than SDR for cases with limited connectivity.
Nonetheless, MDS is faster \cite{galp2023} and has already been applied to various fields \cite{mds_survey}.

\noindent\textbf{Our contributions}
\begin{itemize}
    \item In Section \ref{non_convex_sufficient_condition}, we conduct a comprehensive analysis of the SNL loss function landscape. 
    Our main theorem \ref{main_theorem} establishes that, with a high probability dependent on the number of sensors, the loss function is non-convex. 
    We also demonstrate that directly augmenting the problem's dimension does not yield a convex loss function. 
    However, the second-order dimension augmentation, in other words, the SDR method, is always convex.
    Together with Nie's fourth-order dimension augmentation \cite{nie2009sum}, we can highlight the blessing of high-order dimension augmentation.

    \item It was thought that more edges (information) contribute to a better loss function of SNL and if sufficient edges are given, the loss function will be convex.
    This is correct intuitively because more information usually yields better problems and if sufficient edges are given, the graph of the SNL problem can even be a complete graph, meaning that the SNL problem has the best rigidity.
    However, more edges don't always contribute to a better loss function and even if all edges are given, the loss function can still be non-convex as we will show.
    Specifically, we will provide an example where if more edges are given, the originally convex loss function will become non-convex.
    Additionally, following our main theorem, we will prove that when anchors are relatively few, even if all edges are given, the loss function is non-convex with high probability.
    
    \item Prior works (e.g., \cite{biswas_five_people_2006, biswas_lian_wang_ye_2006, wang2008}) have reported practical success using the SDR solution as a warm-start for minimization, but a comprehensive theoretical explanation for this success has been lacking. 
    Building on a deeper understanding of the loss function landscape, we provide a more rigorous explanation for the effectiveness of this method, referred to as SDR+GD in our paper. 
    The explanation draws upon insights from graph rigidity theory, offering valuable insights into the reasons behind the success of SDR+GD.

    \item Some perspectives are provided in this paper.
    First and most importantly,  the success of the second-order and the fourth-order dimension augmentation and the failure of the first-order dimension augmentation in the context of SNL, and the success of the second-order dimension augmentation in the context of max-cut indicates the blessing of high-order dimensionality.
    Second, the advantage of specific dimension augmentation and the advantage of the scenario that uses the high-dimensional solution as the warm-start of the new model are observed in both the SNL and the neural network fields.
    The reasons for the advantages in the SNL field and in the neural networks field may have common explanations.
\end{itemize}

\noindent\textbf{Outline of the paper}

The rest of the paper is organized as follows. 
In Section \ref{preliminaries} we define some key concepts. 
In Section \ref{Non-Convexity of the Low-Dimensional Optimization} we prove that the loss function \eqref{loss function with parameter bc} is non-convex with high probability in our main Theorem \ref{main_theorem} and provide an example where as more edges are given, the originally convex loss function becomes non-convex.
We also show that direct dimension augmentation cannot improve the loss function method. 
In Section \ref{Convexity of the High-Dimensional Optimization} we discuss the convexity of the SDR's optimization problem obtained by a product-type dimension augmentation of the loss function method. 
In Section \ref{High-dimensional Solution as a Warm-Start of the Low-dimensional Optimization} we provide an explanation for why the regularized SDR+GD method in \cite{biswas_five_people_2006,biswas_lian_wang_ye_2006} works.
In Section \ref{Comparison to the Neural Network}, we discuss the relation between our findings and some similar methods in neural networks, including dimension augmentation and projecting warm-starts.
Finally, in Section \ref{conclu}, we give our conclusions and some future work. 

\section{Preliminaries}
\label{preliminaries}
We will define the concepts and notations that we will discuss in this paper. 
The notation in this paper is mostly standard.

First of all, we use $\Vert\cdot\Vert$ to denote the $L_2$ norm and $\Vert\cdot\Vert_1$ to denote the $L_1$ norm.
Moreover, the distance we refer to in this paper is the Euclidean distance.
We use $\overline{A}$ to denote the complement of a set $A$ and $\sqcup$ to denote the disjoint union.
Given symmetric matrices $X$ and $Y$, we use $X\succeq Y$ to indicate that $X-Y$ is positive semi-definite.

\textit{Graph} $G=(V,E)$ in this paper is first a finite, simple, and undirected graph. 
Additionally, the vertices of it are divided into two types of vertices or two colors. 
One is ``anchor''. 
The other is ``sensor''. 
Finally, there is no edge between anchors and anchors.
In this paper, we always use $s_i$ to denote sensors and $a_i$ to denote anchors.
Therefore, if $G$ has $n$ sensors and $m$ anchors, we have $G$'s vertices set $V=\{s_1,\dots,s_n,a_1,\dots,a_m\}$.
We also use $V_s$ to denote all sensors and $V_a$ to denote all anchors.
Therefore, we have $V_s=\{s_1,\dots,s_n\}$ and $V_a=\{a_1,\dots,a_m\}$.

\textit{Framework} is defined to be a graph $G$ together with an embedding $p:V\to \mathbb{R}^d$ mapping the vertices into an Euclidean space.
Every vertex in $V$ is mapped to its position by the map $p$.
Therefore, we can calculate every Euclidean distance between two vertices.

$N_x, N_a$ are defined to be the known edges of a graph $G$.
They are a bipartition of the set $E$ of $G$'s edges.
$N_x$ is the set of all edges between different sensors; $N_a$ is the set of all edges between sensors and anchors.
Therefore, if $V=\{s_1,\dots,s_n,a_1,\dots,a_m\}$, we have 
$N_x=\{\{s_i,s_j\}|\{s_i,s_j\}\in E\}$ 
and 
$N_a=\{\{a_i,s_j\} | \{a_i,s_j\}\in E\}$.
To simplify notation, we can rewrite 
$N_x=\{\{i,j\}|\{s_i,s_j\}\in E\}$
and 
$N_a=\{(i,j)|\{a_i,s_j\}\in E\}$.
Here, ``$\{\cdot\}$'' means an unordered set and ``$(\cdot)$'' means an ordered set.

An $\textit{SNL problem}$ $(G,d,\tilde{d},a)$ is defined to be a problem where we know $\xi=(G,d,\tilde{d},a)$ and want to find frameworks with proper positions of sensors. 
Among them, $G$ is a graph of sensors and anchors, $d$ is a mapping $N_x\to \mathbb{R}^+$ meaning all given distances between different sensors, $\tilde{d}$ is a mapping $N_a\to \mathbb{R}^+$ meaning all given distances between sensors and anchors, and $a$ is a mapping $V_a\to \mathbb{R}^d$ meaning all positions of the anchors.
To simplify notation, we use $d_{ij}$ to denote $d(\{s_j,s_j\})$, $\tilde{d}_{ij}$ to denote $\tilde{d}(\{a_i,s_j\})$, and the same symbol $a_i$ to denote the position of the anchor $a_i$.
We call an SNL problem $d$-dimensional if the dimension of its anchors is $d$.
We define a framework to be the \textit{solution} to the SNL problem if the framework has the same graph $G$ and the same positions of anchors as the SNL problem and generates equal distances to the SNL problem.
Equivalently, to solve an SNL problem is to find $x_i$'s that satisfy the following equations,
\begin{equation}
\label{solve_the_SNL_problem}
\begin{cases}
    \Vert a_i - x_j\Vert = \tilde d_{ij}\quad \forall (i,j) \in N_a, \\ 
    \Vert x_i - x_j\Vert = d_{ij} \quad \forall \{i,j\} \in N_x.
\end{cases}
\end{equation}
Therefore, we also say that $(x_1^T\ x_2^T\ \dots\ x_n^T)^T\in \mathbb{R}^{dn}$ is a solution to the SNL problem.
Furthermore, a framework can generate a unique SNL problem to which the framework is a solution.

Directly increasing or augmenting the dimension of the SNL problem $(G,d,\tilde{d},a)$ from $d$-dimension to $d'$-dimension means increasing the dimension of the anchors $a$ by adding zeros in the augmented coordinates of the anchors' positions.
Therefore, the solution to the high-dimensional SNL problem has the same dimension as the high-dimensional anchors.
This is what we call a high-dimensional solution to an SNL problem.
Relatively speaking, the solution to the original SNL problem is low-dimensional.

A $\textit{Unit-disk SNL case}$ $(U,r,n,a)$ is defined by the region $U\subseteq \mathbb{R}^d$, the radius $r>0$, the number of sensors $n\in \mathbb{N}$, and the positions of anchors $a={a_1,\dots, a_m}\subset U$. 
This case represents a distribution of the SNL problem $(G,d,\tilde{d},a)$, generated by a random framework. 
The sensors are $n$ random vectors, independently drawn from a probability distribution $P$ on $U$. 
The anchors remain the same as the given ones and an edge is included in the framework if and only if the Euclidean distance between the two vertices of the edge is less than or equal to $r$.
In this paper, $U$ is bounded and $P$ is the uniform distribution on $U$.

The $\textit{loss function}$ of an SNL problem $(G,d,\tilde{d},a)$ is defined to be 
\begin{equation}
\label{loss function with parameter bc}
\operatorname{loss}\left(x\right)=\sum_{\{i,j\} \in N_x}|\Vert x_i-x_j\Vert^b-d_{i j}^b|^c+\sum_{(i, j) \in N_a}|\Vert a_i-x_j\Vert^b-\tilde{d}_{ij}^b|^c,
\end{equation}
where $x_i\in \mathbb{R}^d$ for $i=1,2,\dots,n$.
We call $b$ the inside degree and $c$ the outside degree.
The domain of $\operatorname{loss}(x)$ is $\mathbb{R}^{dn}$ and the codomain of $\operatorname{loss}(x)$ is $\mathbb{R}$. 
Moreover, we define the loss function of a framework to be the loss function of the SNL problem generated by the framework.

In this paper, we often use the loss function
\begin{equation}\label{loss_function}
\operatorname{loss}\left(x\right)=\sum_{\{i,j\} \in N_x}|\Vert  x_i-x_j\Vert ^2-d_{i j}^2|+\sum_{(i, j) \in N_a}|\Vert a_i-x_j\Vert ^2-\tilde{d}_{ij}^2|,
\end{equation}
which is the $b=2,c=1$ case of \eqref{loss function with parameter bc}.

We introduce some basic definitions from graph rigidity theory.

A framework is defined to be \textit{locally rigid} if it is not flexible; i.e., there exists a neighbourhood of the framework's sensors such that there is no other solution to the SNL problem generated by the framework in the neighbourhood.

A framework is defined to be \textit{globally rigid} if the SNL problem generated by the framework has a unique solution.

A framework is defined to be \textit{universally rigid} if the SNL problem generated by the framework has a unique solution in any dimension.

\section{Non-Convexity of the Low-Dimensional Optimization}
\label{Non-Convexity of the Low-Dimensional Optimization}
In this section, we mainly discuss the low-dimensional optimization method, i.e., minimizing the loss function.

This method of minimizing the loss function is widely used in former works such as \cite{biswas_five_people_2006,biswas_lian_wang_ye_2006, tseng_2007, pong_tseng_2010, wang2008}.
In this section, we will show that the loss function is often non-convex. 
Specifically, in Subsection \ref{common_non_convexity}, we will prove that in the unit-disk SNL case, if the anchors are relatively few, the loss function \eqref{loss function with parameter bc} will be non-convex with a high probability, which is our main depiction of the non-convexity of the loss function. 

The unit-disk SNL case is important in practice since the wireless sensors can only communicate with their neighbors within a limited radio range or ``radius". 
The importance of the unit-disk SNL case was also mentioned in \cite{biswas_five_people_2006, biswas_lian_wang_ye_2006, zhu2010universal, wang2008, by2004, kuhn2003ad}.
Hence, it is common that many researchers assume the unit-disk SNL case to carry out probabilistic and combinatorial results on sensor networks, we refer the readers to \cite{Rigidity_computation_and_randomization_in_network_localization, penrose_1999, shamsi_taheri_zhu_ye_2013, bettstetter2002minimum} for details of these results.

We first show some landscapes of the SNL loss function to let the readers know its common non-convexity intuitively in Subsection \ref{landscape_of_the_SNL_loss}.
Next, we introduce our main theorem in Subsection \ref{common_non_convexity}.
Moreover, we will show that direct dimension augmentation is also often non-convex in Subsection \ref{Non-Convexity of Direct Dimension Augmentation}.

\subsection{Landscape Examples of the SNL Loss Function}
\label{landscape_of_the_SNL_loss}

In this subsection, in order to show the landscape more clearly, the dimension of the SNL problem is set to $2$ and the loss function \eqref{loss function with parameter bc} is set to the $b=2,c=1$ case \eqref{loss_function}.

The loss function \eqref{loss_function} can be regarded as a summation of positive terms assigned for each edge of the graph. 
We first consider one term of the summation, i.e., 
\begin{equation}\label{one_term_loss_function}
    f_{ij}(x_j)=|\Vert a_i-x_j\Vert ^2-\tilde{d}_{ij}^2|.
\end{equation}

\begin{figure}[H]
   \begin{minipage}{0.44\textwidth}
     \centering
     \includegraphics[width=\linewidth]{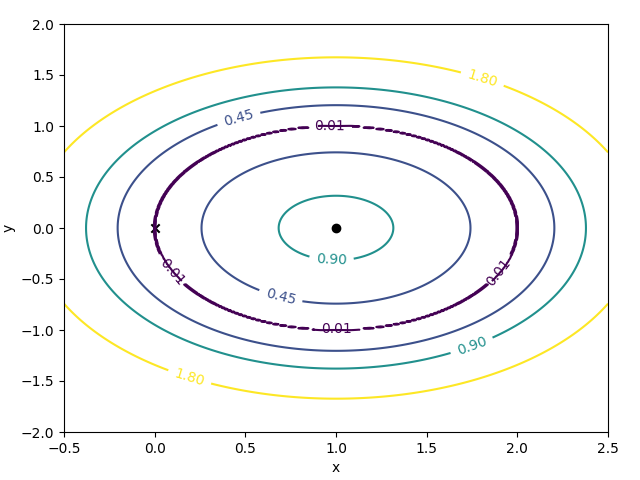}
   \end{minipage}\hfill
    \begin{minipage}{0.55\textwidth}
     \centering
     \includegraphics[width=\linewidth]{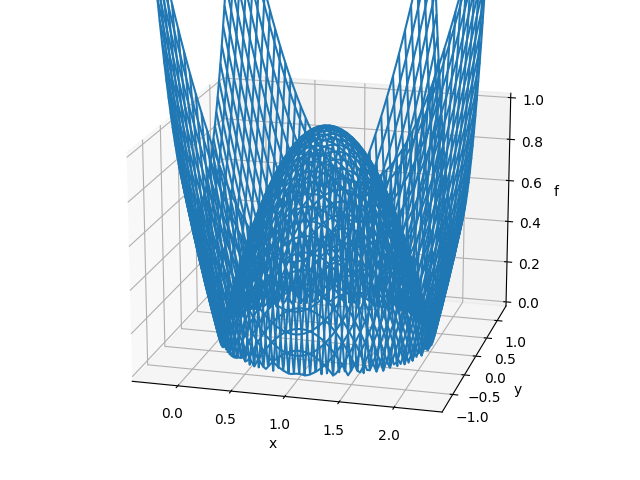}
   \end{minipage}\hfill
   
   \caption{Landscape of loss function of 1 anchor 1 sensor}\label{1a1s}
\end{figure}
The minimizers of the loss function \eqref{one_term_loss_function} form a circle of radius $\tilde{d}_{i j}$ around the anchor, see Figure \ref{1a1s}.

In general, the loss function \eqref{loss_function} is non-convex.
Here is a non-convex example of the SNL loss function with 3 anchors and 1 sensor. 
We construct the example by setting the positions of the anchors to $(-1,0)$, $(1, 0)$ and $(0, 0.4)$ and the position of the sensor to $(0, 1)$.
The graph of the framework is fully connected.
By direct calculation, the loss function of this SNL problem is
$$
\begin{aligned}
    f(x) =& |\Vert x - (-1, 0)\Vert ^2 - (\sqrt 2)^2| + |\Vert x - (1, 0)\Vert ^2 - (\sqrt 2)^2| \\ 
    &+ |\Vert x - (0, 0.4)\Vert ^2 - 0.6^2|, 
\end{aligned}
$$
where $x\in \mathbb{R}^2$.
It is easy to see from Figure \ref{Landscape of loss function of 3 anchors 1 sensor}, which is the landscape of $f(x)$, that $f(x)$ is non-convex and has at least two local minimizers. 
We note that the framework is globally rigid, which means that there exists a unique solution to the SNL problem \eqref{solve_the_SNL_problem}.
However, the loss function $f(x)$ has at least two local minimizers.
As a result, global rigidity does not imply the convexity of the SNL loss function, even if the associated graph of the framework is fully connected. 
Furthermore, as a comparison, we can see in Figure \ref{Landscape of SDR loss function of 3 anchors 1 sensor} that the loss function of SDR optimization is convex.  
\begin{figure}[H]
   \begin{minipage}{0.44\textwidth}
     \centering
     \includegraphics[width=\linewidth]{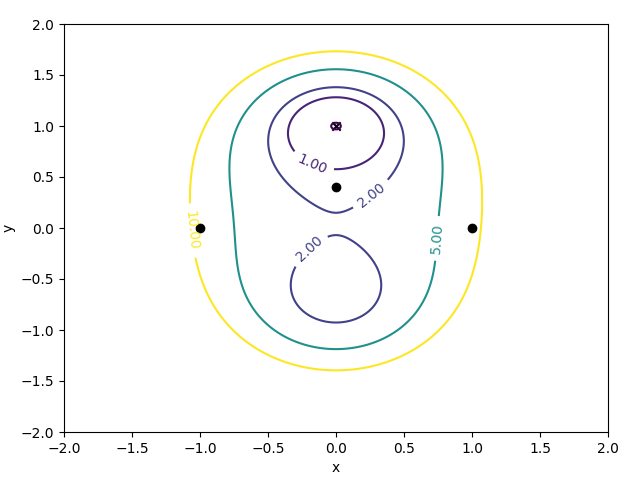}
   \end{minipage}\hfill   
   \begin{minipage}{0.55\textwidth}
     \centering
     \includegraphics[width=\linewidth]{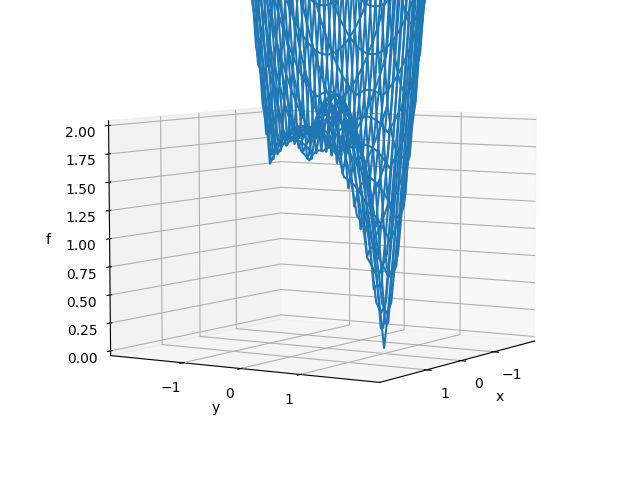}
   \end{minipage}
   
   \caption{Landscape of loss function of 3 anchors 1 sensor}
   \label{Landscape of loss function of 3 anchors 1 sensor}
\end{figure}

\subsection{Common Non-Convexity}\label{common_non_convexity}
In this subsection, we will prove that the loss function is often non-convex in the unit-disk SNL case.

We give an outline of the proof in this paragraph. 
First, a sufficient condition for the loss function \eqref{loss function with parameter bc} to be non-convex is given in Lemma \ref{lemma_of_main_theorem_sufficient_condition_of_non_convex}.
Then we use the sufficient condition \eqref{non_convex_sufficient_condition} to estimate the probability.
We first estimate the left-hand side of \eqref{non_convex_sufficient_condition} by a constructive event.
The event is an intersection of events that control the lower bound on the distances on the left-hand side of \eqref{non_convex_sufficient_condition}.
The probability of the event is estimated by some probability inequalities.
Using some deterministic inequalities, we estimate the right-hand side of \eqref{non_convex_sufficient_condition}.
Then if this event occurs, the sufficient condition \eqref{non_convex_sufficient_condition} holds.
This finishes the proof.

Then we begin the rigorous proof.

We first have this key lemma as a sufficient condition for the loss function \eqref{loss function with parameter bc} to be non-convex.
\begin{lem}\label{lemma_of_main_theorem_sufficient_condition_of_non_convex}
A sufficient condition for the loss function \eqref{loss function with parameter bc} to be non-convex is
\begin{equation}\label{non_convex_sufficient_condition}
2\sum_{\{i,j\}\in N_x}d_{ij}^{bc}>\sum_{(i,j)\in N_a}|\Vert a_i+s_j\Vert ^b-\tilde{d}_{ij}^b|^c.
\end{equation}
\end{lem}
\begin{proof}
    Consider three points.
    $s\in \mathbb{R}^{dn}$ is the sensors' point.
    $z=0\in \mathbb{R}^{dn}$ is the second point.
    $y=-s$ is the third point.
    It is obvious that $z=0.5s+0.5y$.
    Therefore, a sufficient condition for the loss function \eqref{loss function with parameter bc} to be non-convex is 
    $$
    2\operatorname{loss}(z)>\operatorname{loss}(s)+\operatorname{loss}(y),
    $$
    which is 
    $$
    2\sum_{\{i,j\}\in N_x}d_{ij}^{bc}>\sum_{(i,j)\in N_a}(|\Vert a_i+s_j\Vert ^b-\tilde{d}_{ij}^b|^c-2|\Vert a_i\Vert ^b-\tilde{d}_{ij}^b|^c).
    $$
    Therefore, the inequality \eqref{non_convex_sufficient_condition} is the sufficient condition for the loss function \eqref{loss function with parameter bc} to be non-convex.
    This finishes the proof.
\end{proof}
\begin{rem}
    The condition appears to be too strong in general. 
    However, the condition is not that strong in the specific context of SNL.
    As demonstrated in Theorem \ref{qualitative_sufficient_two_local_minimizers}, the non-convexity of the loss function \eqref{loss function with parameter bc} typically arises from the symmetry of the nodes.
    Additionally, we can translate and rotate the nodes such that the axis of symmetry passes the origin.
    Therefore, this sufficient condition is the key of the non-convexity of the loss function \eqref{loss function with parameter bc} to some extent.
\end{rem}

\begin{thm}\label{main_theorem}
    For a unit-disk SNL case $(U,r,n,a)$, assume that the region $U$ is connected, open, and bounded.
    Then for any given set of anchors $a$, radius $r>0$, and probability $p_0<1$, there exists $N\in\mathbb{N}$ such that for all number of sensors $n>N$, the probability that the loss function \eqref{loss function with parameter bc} of the SNL problem in the unit-disk SNL case $(U,r,n,a)$ is not convex is larger than $p_0$.
    Specifically, a lower bound on the probability is $1-nv\tbinom{n-1}{v-1}(1-p_r)^{n-v}$.
    Here, $v=\lfloor \sqrt{n} \rfloor$
    and $p_r$ is a positive constant probability dependent solely on the region $U$ and the radius $r$.
    The definition of $p_r$ will be given in the proof.
\end{thm}

\begin{proof}
Without loss of generality, we assume that the origin is in the region $U$.

Then we define some events.
Define event $A_i$ to be $\{$There exist at least $v$ sensors in the region between two concentric circles with radii $\varepsilon r$ and $r$ centered at $s_i\}$. 
Here, we set $v$ to $\lfloor \sqrt{n} \rfloor$ and set $\varepsilon\in (0,1)$ to satisfy the following property: no matter where the sensor $s$ is, the probability that the distance from $s$ to another independent random sensor $s'$ belongs to $[\varepsilon r,r]$ has a uniform lower bound $p_r>0$, or more rigorously, there exists $p_r>0$ such that for all $x\in U$, we have
    $$
        P(\Vert s-s'\Vert \in[\varepsilon r,r]|s=x) \geq p_r.
    $$
    In most cases, such as a square region, we can simply set $\varepsilon$ to $0.5$.

    If the event $A_1\cap A_2\cap\dots\cap A_n$ occurs, we have
    \begin{equation}\label{high_probability_not_one_point_convex_s_i-s_j}
    \sum_{\{i,j\}\in N_x}d^{bc}_{ij}\geq \frac{1}{2}nvr^{bc}\varepsilon^{bc}.
    \end{equation}

    Then we estimate the probability $P(A_1\cap A_2\cap\dots \cap A_n)$.
    By basic knowledge of probability theory, we have
    \begin{align}
        P(A_1\cap A_2\cap\dots \cap A_n)&=1-P(\overline{A_1\cap A_2\cap\dots\cap A_n})\notag \\
        &=1-P(\overline{A_1}\cup \overline{A_2}\cup\dots\cup \overline{A_n})\notag \\
        &\geq 1-\sum_{i=1}^nP(\overline{A_i})\notag \\
        &=1-nP(\overline{A_1}).\label{one_minus_np}
    \end{align}
    
    Then we estimate $P(\overline{A_1})$. 
    Define event $A_{1i}$ to be $\{$There exist exactly $i$ sensors in the region between two concentric circles with radii $\varepsilon r$ and $r$ centered at $s_1\}$.
    Then, we have
    $$
    \overline{A_1}=A_{10}\sqcup A_{11}\sqcup A_{12}\sqcup\dots\sqcup A_{1(v-1)}.
    $$
    Therefore, we have
    \begin{equation}\label{sum_of_probability}
    P(\overline{A_1})=P(A_{10})+P(A_{11})+\dots+P(A_{1(v-1)}).
    \end{equation}
    Then we estimate $P(A_{1k})$ for $k=0,\dots, v-1$. 
    Let $s'$ be an independent random sensor, we define
    $$
        p_{s_1}: = P({\Vert s'-s_1 \Vert \in [\varepsilon r, r]}|s_1)
    $$
    for all fixed $s_1\in U$. 
    By the definition of $p_r$, we have $p_{s_1} \in [p_r,1]$. 
    Therefore, 
    \begin{equation}
    \label{p(A1k)}
        P(A_{1k})=\mathbb E(\mathbbm{1}_{A_{1k}})=\mathbb E(\mathbb E(\mathbbm{1}_{A_{1k}}|s_1))=\mathbb E(\tbinom{n-1}{k}(1-p_{s_1})^{n-1-k}p_{s_1}^k),
    \end{equation}
    for all $k=0,1,\dots,v-1$.
    Then we estimate the term $\tbinom{n-1}{k}(1-p_{s_1})^{n-1-k}p_{s_1}^k$ in Lemma \ref{lemma_by_derivation}.
\begin{lem}\label{lemma_by_derivation}
If $n$ satisfies $(n-1)p_r\geq \lfloor \sqrt{n} \rfloor - 1$ and $k\in\{0,1,\dots,v-1\}$ is given, we have 
    $$
    (1-p)^{n-1-k}p^k\leq(1-p_r)^{n-1-k}p_r^k
    $$
    for all $p\in[p_r,1]$.
\end{lem}
\begin{proof}
        Define 
        $$
        f(p):=(1-p)^{n-1-k}p^k.
        $$
        Then its derivative is
        \begin{equation}
        \label{probability_derivative}
        f'(p)=(1-p)^{n-1-k-1}p^{k-1}\big( (n-1-k)(-1)p+(1-p)k \big).
        \end{equation}
        
        As $(n-1)p_r\geq \lfloor \sqrt{n} \rfloor - 1$, we have 
        $$
        k\leq (n-1)p
        $$
        for all $k=0,1,2,\dots,v-1$ and $p\in[p_r,1]$.
        Therefore, we have
        $$
        (k+1-n)p+(1-p)k = (1-n)p+k\leq0
        $$
        for all $k=0,1,2,\dots,v-1$ and $p\in[p_r,1]$.
        Therefore, according to \eqref{probability_derivative}, we have $f'(p)\leq 0$ for all $p\in [p_r, 1]$.
        As a result, $f(p)\leq f(p_r)$; i.e.,
    $$
    (1-p)^{n-1-k}p^k\leq(1-p_r)^{n-1-k}p_r^k.
    $$
    \end{proof}
    
    By Lemma \ref{lemma_by_derivation}, if $n$ is large enough, we have
    $$
    P(A_{1k}) = \mathbb E(\tbinom{n-1}{k}(1-p_{s_1})^{n-1-k}p_{s_1}^k)\leq \tbinom{n-1}{k}(1-p_r)^{n-1-k}p_r^k.
    $$
    Therefore, by \eqref{sum_of_probability}, we have
    \begin{align}
    &P(\overline{A_1})\notag\\
    =&P(A_{10})+P(A_{11})+\dots+P(A_{1(v-1)})\notag\\
    \leq& \sum_{k=0}^{v-1}\tbinom{n-1}{k}(1-p_r)^{n-1-k}p_r^k\notag \\
    \leq& v\tbinom{n-1}{v-1}(1-p_r)^{n-v}\notag.
    \end{align}
    As a consequence, for large enough $n$, 
    $$
    nP(\overline{A_1})\leq nv\tbinom{n-1}{v-1}(1-p_r)^{n-v}\leq vn^{\sqrt{n}}(1-p_r)^{\frac{1}{2}n}.
    $$
    As $n$ becomes larger, $vn^{\sqrt{n}}(1-p_r)^{\frac{1}{2}n}$ is approaching $0$. 
    Therefore, by \eqref{one_minus_np}, $P(A_1\cap A_2\cap\dots \cap A_n)$ is approaching $1$.

    Next, we consider \eqref{non_convex_sufficient_condition}. Denote the region's diameter by $D$; i.e., 
    $$
    D:=\sup_{x,y\in U}\Vert x-y\Vert .
    $$
    Since we have assumed that the origin is in the region $U$, we have $\Vert a_i\Vert ,\Vert s_i\Vert \leq D$.
    Due to the definition of $D$, we also have $|\tilde{d}_{ij}|\leq D$.
    Therefore, we have
    \begin{align}
        &\sum_{(i,k)\in N_a}|\Vert a_i+s_j\Vert ^b-\tilde{d}_{ij}^b|^c \notag\\ 
        \leq&\sum_{(i,k)\in N_a}\max\{\Vert a_i+s_j\Vert ^{bc},\tilde{d}_{ij}^{bc}\} \notag\\ 
        \leq&\sum_{(i,k)\in N_a}(2D)^{bc} \notag \\ 
        \leq&nm(2D)^{bc}.\label{high_probability_not_one_point_convex_N_a_term}
    \end{align}
    Here, we denote the number of anchors by $m$.
    Due to \eqref{high_probability_not_one_point_convex_s_i-s_j} and \eqref{high_probability_not_one_point_convex_N_a_term}, if $n$ is big enough, \eqref{non_convex_sufficient_condition} holds.
    Therefore, with a high probability larger than $1-nv\tbinom{n-1}{v-1}(1-p_r)^{n-v}$, the loss function \eqref{loss function with parameter bc} of the SNL problem in the unit-disk SNL case is not convex.
\end{proof}
\begin{cor}
    If we let $r>\sqrt{2}$, the graph of the random framework will always be complete.
    This means that even if all edges are given, when the anchors are relatively few, the loss function \eqref{loss function with parameter bc} is non-convex with high-probability in the unit-disk SNL case.
\end{cor}

Following the above corollary, we provide an example to show further that more edges (information) don't always imply better convexity of the loss function \eqref{loss function with parameter bc}.

Actually, we provide a stronger example here.
The example satisfies that as the radius of the sensors increases, which implies that more edges are given, the originally convex loss function \eqref{loss_function} becomes non-convex.

In this example, the anchors are respectively $a_1=(-e,h)$, $a_2=(-e,-h)$, $a_3=(e,h)$, $a_4=(e,-h)$, and the sensors are respectively $s_1=(-e,0)$ and $s_2=(e,0)$.
Here, $e>0$ and $2e-\varepsilon<h<2e$ for a very small $\varepsilon>0$.
The connections among the nodes are shown in Figure \ref{Connection among the nodes}. 
\begin{figure}[H]
   \begin{minipage}{0.5\textwidth}
     \centering
     \includegraphics[width=\linewidth]{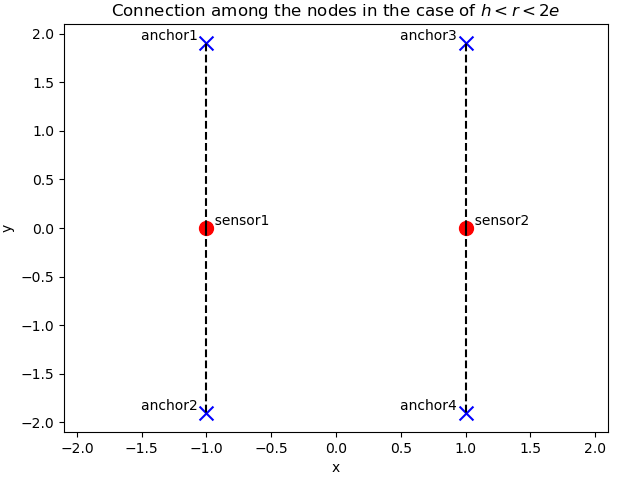}
   \end{minipage}\hfill   
   \begin{minipage}{0.5\textwidth}
     \centering
     \includegraphics[width=\linewidth]{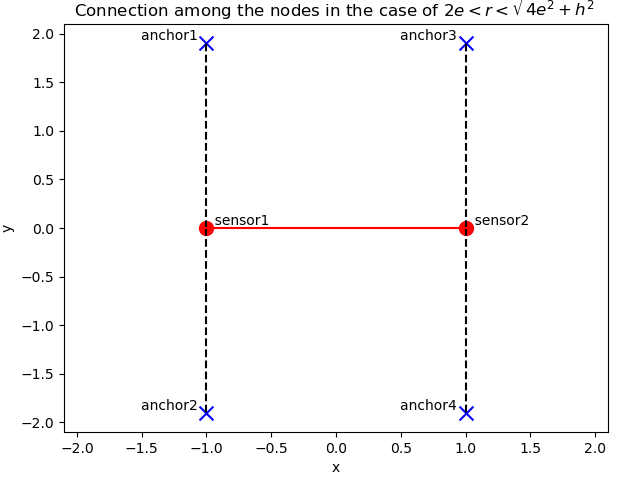}
   \end{minipage}
   
   \caption{The connections among the nodes in the case of $h<r<2e$ (left) and $2e<r<\sqrt{4e^2+h^2}$ (right).}
   \label{Connection among the nodes}
\end{figure}

When the radius $r$ is relatively small, the framework contains four edges, which are $\{a_1,s_1\}$, $\{a_2,s_1\}$, $\{a_3,s_2\}$ and $\{a_4,s_2\}$.
In this case, we denote the loss function by
\begin{equation*}
    F_0(x_{11},x_{12},x_{21},x_{22})=f_1(x_{11},x_{12})+f_2(x_{21},x_{22}),
\end{equation*}
where $f_1$ denotes the two terms of edges $\{a_1,s_1\}$ and $\{a_2,s_1\}$, and $f_2$ denotes the two terms of edges $\{a_3,s_2\}$ and $\{a_4,s_2\}$.
To prove that $F_0$ is convex, we only need to prove that $f_1$ and $f_2$ are convex.
To prove that $f_1$ and $f_2$ are convex, we only need to prove that the function
$$
f(x,y)=|x^2+(y-h)^2-h^2|+|x^2+(y+h)^2-h^2|
$$
is convex, since $f_1$ and $f_2$ can be obtained by translation of $f$.
Additionally, by direct calculation, we have 
$$
f(x,y)=\max\{x^2+(y-h)^2-h^2+x^2+(y+h)^2-h^2,4hy,-4hy\}.
$$
Since $f$ is the maximum of three convex functions, $f$ itself is a convex function.
As a conclusion, the loss function $F_0$ is a convex function.

When the radius $r$ is relatively large, the frame work contains five edges, which are $\{a_1,s_1\}$, $\{a_2,s_1\}$, $\{a_3,s_2\}$, $\{a_4,s_2\}$ and $\{s_1,s_2\}$.
In this case, we denote the loss function by $F(x_{11},x_{12},x_{21},x_{22})$.
Then we can directly verify that 
$$
F(-e,h,e,h)+F(e,h,-e,h)<2F(0,h,0,h).
$$
Therefore, $F$ is non-convex.

Moreover, we will introduce another theorem to show that in some cases, the loss function \eqref{loss function with parameter bc} has more than one local minimizers.
\begin{thm}
\label{qualitative_sufficient_two_local_minimizers}
    Let $\rho$ be a $2$-dimensional locally rigid framework with colinear anchors.
    Consider a new framework $\tau$ obtained by slightly perturbing the anchors of $\rho$.
    Then, the loss function \eqref{loss function with parameter bc} of framework $\tau$ still possesses at least 2 local minimizers.
\end{thm}

\begin{proof}
    Apparently, there are at least $2$ global minimizers of $\rho$'s loss function.
    One is the true solution, which is denoted by $y_1$.
    The other is the mirror image of the true solution with respect to the line where anchors lie, which is denoted by $y_2$.
    Because $\rho$ is locally rigid, the two global minimizers are isolated.
    Therefore, there exist two small disjoint balls centered at $y_1$ and $y_2$.
    The two balls' boundary function values have a strictly positive infimum.
    We denote the ball centred at $y_2$ by $B_2$.
    Upon applying a slight perturbation, both the function value at $y_2$ and the infimum of the function values at $\partial B_2$ continuously change. 
    Therefore, if the perturbation is sufficiently small, the function value at $y_2$ will remain smaller than the infimum of the function values at $\partial B_2$.
    This implies the presence of at least one local minimizer within $B_2$.
    Coupled with the persistent global minimizer $y_1$, we can conclude that there are still at least two local minimizers of the loss function of the new framework $\tau$.
\end{proof}
This theorem suggests that for the framework whose anchors are almost colinear, its loss function exhibits unfavorable properties. 
Actually, the almost colinearity of the sensors can contribute to more local minimizers.
In numerous cases, some of the anchors and sensors are almost colinear, which contributes to the common bad loss function.

\subsection{Non-Convexity of Direct Dimension Augmentation}
\label{Non-Convexity of Direct Dimension Augmentation}
Recall that the domain and codomain of the loss function \eqref{loss function with parameter bc} are
$$
\operatorname{loss}\left(x\right):\mathbb{R}^{dn}\to \mathbb{R}.
$$
If we directly augment its dimension to $n(n+d)$, the new loss function $\tilde{\operatorname{loss}}:\mathbb{R}^{n(n+d)}\to \mathbb{R}$ will be
\begin{equation}
\label{new_loss_function_with_augmented_dimension}
    \tilde{\operatorname{loss}}\left(x\right)=\sum_{\{i,j\} \in N_x}|\Vert  x_i-x_j\Vert ^b-d_{i j}^b|^c+\sum_{(i, j) \in N_a}|\Vert a_i-x_j\Vert ^b-\tilde{d}_{ij}^b|^c,
\end{equation}
where $x_i\in \mathbb{R}^{n+d}$ are the variables.
We augment the loss function to $(n+d)$-dimension because the SDR method finds the high-dimensional solution in $(n+d)$-dimension.
Theorem \ref{main_theorem} also holds for this new loss function.
This suggests that the new loss function \eqref{new_loss_function_with_augmented_dimension} is also often non-convex.

\begin{thm}\label{main_theorem_high_dimensional}
    For a unit-disk SNL case $(U,r,n, a)$, assume that the region $U$ is connected, open, and bounded.
    Then for any given set of anchors $a$, radius $r>0$, and probability $p_0<1$, there exists $N\in\mathbb{N}$ such that for all number of sensors $n>N$, the probability that the new loss function \eqref{new_loss_function_with_augmented_dimension} of the SNL problem in the unit-disk SNL case $(U,r,n,a)$ is not convex is larger than $p_0$.
    Specifically, a lower bound on the probability is $1-nv\tbinom{n-1}{v-1}(1-p_r)^{n-v}$.
    The definitions of $v$ and $p_r$ remain consistent with those in Theorem \ref{main_theorem}.
\end{thm}
\begin{proof}
    Due to the definition of the new loss function \eqref{new_loss_function_with_augmented_dimension} and Theorem \ref{main_theorem}, there exists $N\in\mathbb{N}$ such that for all number of sensors $n>N$, the probability that the new loss function \eqref{new_loss_function_with_augmented_dimension} of the generated framework is not convex in its $nd$-dimensional subspace is larger than $1-nv\tbinom{n-1}{v-1}(1-p_r)^{n-v}$, so the new loss function is not convex in its entire $n(d+n)$-dimensional space with probability larger than $1-nv\tbinom{n-1}{v-1}(1-p_r)^{n-v}$.
    This finishes the proof.
\end{proof}

The above theorem indicates that the new loss function \eqref{new_loss_function_with_augmented_dimension} is often non-convex, which is similar to the original loss function \eqref{loss function with parameter bc}.
This suggests that directly augmenting the dimension of the loss function may not be an effective way to augment the dimension and extract more information, which indicates the failure of the first-order dimension augmentation.

\section{Convexity of the High-Dimensional Optimization}
\label{Convexity of the High-Dimensional Optimization}

In 2004, Biswas and Ye \cite{by2004} proposed a semi-definite relaxation (SDR) algorithm based on semi-definite programming (SDP) to solve the SNL problem. 
The SDR method is the high-dimensional optimization we mainly discuss in this paper.
Since then, the SDR method has emerged as a promising approach to solving the SNL problem. 

As a product-type dimension augmentation of the former loss function method, the optimization problem of the SDR method is always a convex optimization, which we will show in Subsection \ref{SDR_convex}.
In contrast, the direct first-order dimension augmentation of the loss function is often non-convex, as shown in Subsection \ref{Non-Convexity of Direct Dimension Augmentation}.

\subsection{Formulation of the SDR method}
\label{SDR_convex}
We recall that an SNL problem \eqref{solve_the_SNL_problem} is to find a realization $x_1, \dots, x_n \in \mathbb{R}^d$, such that 
$$
\begin{cases}
    \Vert a_i - x_j\Vert  = \tilde d_{ij}\quad \forall (i,j) \in N_a, \\ 
    \Vert x_i - x_j\Vert  = d_{ij} \quad \forall \{i,j\} \in N_x.
\end{cases}
$$

Denote by $X = (x_1, x_2, \dots, x_n)$ the $d\times n$ matrix of the positions of all sensors.
Let $Y=X^TX$.
Denote by $Y_{ij}$ the $(i,j)$ entry of the matrix $Y$.
We have
\begin{equation}
\label{LinksBetweenCoordinateAndXY}
    \Vert a_i-x_j\Vert ^2 = a_i^Ta_i-2a_i^Tx_j+Y_{jj}, \quad
    \Vert x_i-x_j\Vert ^2 = Y_{ii}-2Y_{ij}+Y_{jj}.
\end{equation}
Equation \eqref{LinksBetweenCoordinateAndXY} suggests that the terms in the loss function \eqref{loss_function} can be rewritten as 
$|Y_{ii}-2Y_{ij}+Y_{jj}-d_{i j}^2|$ or $|a_i^Ta_i-2a_i^Tx_j+Y_{jj}-\tilde{d}_{ij}^2|
$.
Therefore, we can augment the dimension of the original loss function.
A specific form of the dimension augmentation of the loss function \eqref{loss_function} is given by
\begin{align}
\begin{split}
    \text{minimize}\quad &
        \sum_{\{i,j\} \in N_x}|Y_{ii}-2Y_{ij}+Y_{jj}-d_{i j}^2|+\sum_{(i, j) \in N_a}|a_i^Ta_i-2a_i^Tx_j+Y_{jj}-\tilde{d}_{ij}^2|\\
    \text{subject to} \quad &
    Y=X^TX.
\end{split}
\label{y=xtx}
\end{align}
One key step of the semi-definite relaxation (SDR) is relaxing the last equation to $Y\succeq X^TX$. 
Here is the formulation of the relaxed optimization problem 
\begin{equation}
\label{SDRYX}
\begin{aligned}
    \text{minimize}\quad &
        \sum_{\{i,j\} \in N_x}|Y_{ii}-2Y_{ij}+Y_{jj}-d_{i j}^2|+\sum_{(i, j) \in N_a}|a_i^Ta_i-2a_i^Tx_j+Y_{jj}-\tilde{d}_{ij}^2|\\
    \text{subject to} \quad &
    Y\succeq X^TX.
\end{aligned}
\end{equation}

To further develop the problem structure, we introduce the notation:
\begin{equation}
\label{zxy}
Z: = \begin{pmatrix}
I_d & X \\
X^T & Y \\
\end{pmatrix}.
\end{equation}
It is shown in \cite[p.~28]{Boyd94} that $Z \succeq 0$ if and only if $Y - X^TX \succeq 0$.
As a result, the optimization problem \eqref{SDRYX} is equivalent to the following form:
\begin{equation}
\label{SDRZ}
\begin{aligned}
\text{minimize}\quad &
\sum_{{i,j} \in N_x}|Z_{(i+d)(i+d)}-2Z_{(i+d)(j+d)}+Z_{(j+d)(j+d)}-d_{i j}^2| \\
&+\sum_{(i, j) \in N_a}|a_i^Ta_i-2a_i^TZ_{(1:d)j}+Z_{(j+d)(j+d)}-\tilde{d}{ij}^2|\\
\text{subject to} \quad &
Z_{1:d, 1:d} = I_d, \quad Z\succeq 0.
\end{aligned}
\end{equation}
The objective function of \eqref{SDRZ} is convex with respect to $Z$. 
This convexity is demonstrated in Figure \ref{Landscape of SDR loss function of 3 anchors 1 sensor}, which examines the same example as Figure \ref{Landscape of loss function of 3 anchors 1 sensor}.
The true locations of the three anchors are $(-1, 0)$, $(1, 0)$ and $(0, 0.4)$ and the true location of the sensor is $(0, 1)$. 
In this example, $Z$ contains three independent variables: $X\in \mathbb{R}^2$ and $Y\in\mathbb{R}$. 
The constraint in \eqref{SDRZ} can be expressed as $Y\geq X^TX$. 
To visualize the loss landscape of \eqref{SDRZ} over $\mathbb{R}^2$, we fix $X$ and minimize the loss function subject to the given constraint, treating $Y$ as the independent variable. 
The resulting landscape is depicted in Figure \ref{Landscape of SDR loss function of 3 anchors 1 sensor}. 
\begin{figure}[H]
   \begin{minipage}{0.44\textwidth}
     \centering
     \includegraphics[width=\linewidth]{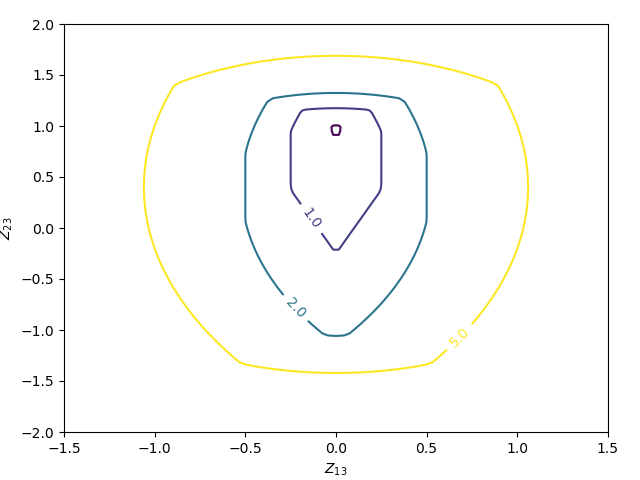}
   \end{minipage}\hfill
   \begin{minipage}{0.55\textwidth}
     \centering
     \includegraphics[width=\linewidth]{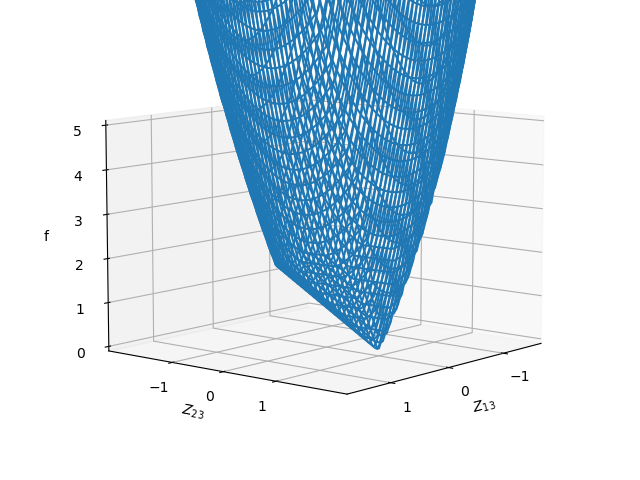}
   \end{minipage}
    \caption{Landscape of SDR loss function of 3 anchors 1 sensor}
    \label{Landscape of SDR loss function of 3 anchors 1 sensor}
\end{figure}
The following paragraphs of this subsection aim to show that the optimization problem \eqref{SDRZ} is equivalent to the optimization problem \eqref{sdr_usual_form}. 
Moreover, the optimization problem \eqref{sdr_usual_form} is a semi-definite programming (SDP) problem.
The SDP problem \eqref{sdr_usual_form} can be readily solved by employing a solver, enabling us to obtain the solution directly. 
We denote
$$
\tilde W_{ij} := a_i^Ta_i-2a_i^Tx_j+Y_{jj}-\tilde{d}_{ij}^2,\quad 
W_{ij}:= Y_{ii}-2Y_{ij}+Y_{jj}-d_{i j}^2.
$$
Denote the $j$th standard basis of $\mathbb{R}^n$ by $e_j$.
By direct calculation, we have
$$
\begin{aligned}
\tilde W_{ij} = 
a_i^Ta_i-2a_i^Tx_j+Y_{jj}-\tilde{d}_{ij}^2 = \begin{pmatrix}
    a_i^T & -e_j^T
    \end{pmatrix}
    Z
    \begin{pmatrix}
    a_i\\
    -e_j
    \end{pmatrix} - \tilde d_{ij}^2, \\
W_{ij}= 
    Y_{ii}-2Y_{ij}+Y_{jj}-d_{i j}^2 = \begin{pmatrix}
            0 & e_i^T-e_j^T
            \end{pmatrix}
            Z
        \begin{pmatrix}
            0\\
            e_i-e_j
        \end{pmatrix} - d_{ij}^2.
\end{aligned}
$$
We denote
$$
\begin{aligned}
    \tilde w_{ij}^+ = \max (\tilde W_{ij}, 0), \quad
    \tilde w_{ij}^- = -\min (\tilde W_{ij}, 0), \\
    w_{ij}^+ = \max (W_{ij}, 0), \quad
    w_{ij}^- = -\min (W_{ij}, 0). 
\end{aligned}
$$
It follows that \eqref{SDRZ} is equivalent to

\begin{align}
\begin{split}
\label{sdr_usual_form}
    \text{minimize}\quad & \sum_{\{i,j\} \in N_x}(w_{ij}^+ + w_{ij}^-)+\sum_{(i, j) \in N_a}(\tilde{w}_{ij}^+ + \tilde{w}_{ij}^-)\\
    \text{subject to}\quad  
    & \begin{pmatrix}
    a_i^T & -e_j^T
    \end{pmatrix}
    Z
    \begin{pmatrix}
    a_i\\
    -e_j
    \end{pmatrix} - \tilde d_{ij}^2 = \tilde w_{ij}^{+} - \tilde w_{ij}^{-}, \quad
    \forall (i,j) \in N_a\\
        & \begin{pmatrix}
            0 & e_i^T-e_j^T
            \end{pmatrix}
            Z
        \begin{pmatrix}
            0\\
            e_i-e_j
        \end{pmatrix} - d_{ij}^2 = w_{ij}^{+} - w_{ij}^{-}, \quad \forall \{i,j\} \in N_x\\
        & w_{ij}^{+}, w_{ij}^{-} \geq 0, \quad \forall (i,j) \in N_a\\
        & \tilde w_{ij}^{+}, \tilde w_{ij}^{-} \geq 0, \quad \forall \{i,j\} \in N_x \\
        & Z_{1:d, 1:d} = I_d,\ Z \succeq 0.
\end{split}
\end{align}

We note that \eqref{sdr_usual_form} is a convex optimization problem.
Specifically, it is a semi-definite programming (SDP) problem.
We can also see similar reformulations of \eqref{solve_the_SNL_problem} as convex optimization problems in \cite{by2004, biswas_five_people_2006, semidefiniterelax, so2007, wang2008}.
In this section, the SNL problem \eqref{solve_the_SNL_problem} is reformulated as the optimization problem \eqref{sdr_usual_form} by delicately increasing the dimension of the problem, leading to convexity, solvability, and other favorable properties of the high-dimensional optimization problem. 

Furthermore, using similar techniques, we can do a second-order dimension augmentation to the optimization problem to minimize the loss function \eqref{loss function with parameter bc} with $b=c=2$, which is a multivariate polynomial. 
This reformulation transforms the problem into an SDP problem with a convex quadratic objective function. 
Consequently, it offers an alternative approach to \cite{lasserre2007sum, lasserre2009moments} for accomplishing a second-order dimension augmentation of the SNL problem's polynomial loss function.

\subsection{Properties of the Solution Matrix}\label{properties of the solution matrix}

A fundamental theoretical result of the SDR method was found by So and Ye in \cite[Theorem 2]{so2007}, which showed that a semi-definite programming solver based on the interior point method finds the true positions of the sensors if and only if the framework is universally rigid.
Additionally, in the unit-disk SNL case, which is common in practice, the probability that the graph is a trilateration and the probability that the framework is universally rigid were estimated by \cite{Rigidity_computation_and_randomization_in_network_localization} and \cite{shamsi_taheri_zhu_ye_2013} respectively.
It was shown in \cite{zhu2010universal, sophd} that if the graph of a framework is a trilateration, the framework will be universally rigid. 
Furthermore, results from \cite{Rigidity_computation_and_randomization_in_network_localization} and \cite{shamsi_taheri_zhu_ye_2013} indicated that it is usual for the SDR method to find the true solution to the SNL problem directly. 

We give the geometric meaning of the solution given by the SDP solver in this paragraph.
We decompose $Z$ to obtain $X$ and $Y$ by equation \eqref{zxy}. 
If $\text{rank}(Z) = d$, by knowledge of linear algebra, $Y=X^TX$.
In this case, following \eqref{y=xtx}, $X$ is a solution to the SNL problem \eqref{solve_the_SNL_problem}. 
If $\text{rank}(Z) \neq d$, our solution is indeed a solution to the relaxed problem \eqref{SDRYX} rather than \eqref{y=xtx}. 
Since $Z\succeq 0$, we have $Y - X^TX \succeq 0$ by knowledge of linear algebra. 
It follows that there exists $0 \neq X_2 \in \mathbb{R}^{n\times n}$ such that $X_2^TX_2 = Y - X^TX$. 
It was suggested by So in \cite[Section 3.4.2]{sophd} that
\begin{equation}
\label{x1x2}
    \begin{pmatrix}
        X \\
        X_2
    \end{pmatrix}
\end{equation}
is a solution to the $(n+d)$-dimensional SNL problem.

If the framework that generates the SNL problem is not universally rigid, we can implement some other methods to get the true solution from the high-dimensional solution that we obtained.
Matrix $Z$ in \eqref{sdr_usual_form} and matrix \eqref{x1x2}, are two forms of the high-dimensional solution. 
We may take $Z_{1:d, (d+1):(d+n)}$, i.e., $X$ in \eqref{x1x2}, as the solution directly. 
In addition, we can do gradient descent using $Z_{1:d, (d+1):(d+n)}$ as a warm-start to minimize the loss function, which will be discussed in Section \ref{High-dimensional Solution as a Warm-Start of the Low-dimensional Optimization}. 
We can also construct new regularization by $Z$ to restart the SDP solver, which was reported to have a large improvement; see \cite{biswas_five_people_2006} for details. 

As a conclusion, as stated in Section \ref{Introduction}, under certain circumstances (universally rigid), the finite convergence (second-order convergence) of the dimension augmentation is guaranteed; i.e., if the framework is universally rigid, the second-order dimension augmentation, in other words the SDR method, will find the same optimal point as the low-dimensional loss function method.

\section{Explanation for the Success of the SDR+GD Method}
\label{High-dimensional Solution as a Warm-Start of the Low-dimensional Optimization}
If the SDR solution is considered as a warm-start of the minimizing loss function method, the probability to find the true low-dimensional solution will be greatly enhanced. 
This phenomenon was observed and stated in \cite{biswas_lian_wang_ye_2006, biswas_five_people_2006}, but only an empirical explanation is provided.

The authors use the SDR solution as the initial value for minimizing the loss function \eqref{loss function with parameter bc} with $b=c=2$. 
To achieve this optimization, they employ the gradient descent algorithm, a widely-used iterative method for minimizing differentiable functions. 
This process involves subtracting the gradient of the current point at each step, with the option of incorporating a step size factor. 
It is important to note that in the case of a non-convex objective function, gradient descent may converge to local minima, underscoring the significance of selecting an appropriate starting point.
Alternatively, other optimization algorithms, like steepest descent, can be employed as well; see, e.g., \cite{wang2008}.

In this section, we will show our more rigorous explanation for the success of the regularized SDR+GD (gradient descent) method in \cite{biswas_five_people_2006, biswas_lian_wang_ye_2006}.

Denote the set of all solutions of an SNL problem $\xi_d$ by $S(\xi_d)$.
If we increase the dimension of the SNL problem $\xi_d$ from $d$ to $d'$, we denote the new high-dimensional SNL problem by $\xi_{d'}$.
We will prove that there exists a simple smooth path with constant edge lengths joining every two solutions of one SNL problem in a higher dimension due to \cite{bezdek2004kneser}.
This means that every point on the simple smooth path is a solution to the high-dimensional SNL problem.

The construction of the simple smooth path was from a lemma which is used to prove the Kneser--Poulsen conjecture in \cite{bezdek2004kneser}.
\begin{thm}
    Suppose that ${p}=({p_1},{p_2},\dots,{p_n})$ and ${q}=({q_1},{q_2},\dots,{q_n})$ are two solutions to a $d$-dimensional SNL problem $\xi_d=(G,d,\tilde{d},a)$. 
    Then the following path ${p}(t):[0,1]\to \mathbb{R}^{2nd}$ joins ${p}$ and ${q}$,
    \begin{align*}
    &{p}_i(t):=\left(\frac{{p}_i+{q}_i}{2}+(\cos \pi t) \frac{{p}_i-{q}_i}{2},(\sin \pi t) \frac{{p}_i-{q}_i}{2}\right), \quad 1 \leq i \leq n,\\
    &p(t):=(p_1(t),p_2(t),\dots,p_n(t)).
    \end{align*}
    Additionally, every point on the path is a solution to the $2d$-dimensional SNL problem $\xi_{2d}$; i.e., $S(\xi_{2d})$ contains the path.
\end{thm}
\begin{proof}
    It is not difficult to directly verify that the path joins ${p}$ and ${q}$, and for all $t\in[0,1]$ and $\{i,j\}\in N_x$, $\Vert {p}_i(t)-{p}_j(t)\Vert $ is constant.
    Similarly, it is not difficult to directly verify that for all $t\in[0,1]$ and $(i,j)\in N_a$, $\Vert a_i- {p}_j(t)\Vert $ is constant.
    Therefore, ${p}(t)$ always satisfies the distance condition.
    This means that ${p}(t)$ is always a solution to the $2d$-dimensional SNL problem $\xi_{2d}$.
\end{proof}
Then for an SNL problem, if the SDR solution ${p}$ is a $d'$-dimensional solution where $d'>d$, there exists a simple smooth path in $S(\xi_{2d'})$ joining the high-dimensional solution ${p}$ and the true solution ${p_0}$. 

As regularization can be included in the SDR method, the solution will minimize the regularized objective function.
In some cases, if the direction along which the regularization decreases is similar to the projection direction, due to the path connectedness, the solution found by SDR may be close to the true $d$-dimensional solution.
Finally, we choose the projection of the high-dimensional solution as the warm-start for the loss function method.
Since the projection is close to the global minimizer, the gradient descent method may avoid converging to other local minimizers of the loss function that are far from the global minimizer.

Specifically, the regularization provided by \cite{biswas_five_people_2006, biswas_lian_wang_ye_2006} demonstrates superior performance compared to other baseline methods.
The key insight lies in the fact that the regularization is designed to maximize the sum of distances between unconnected nodes, thereby minimizing the length in the projection direction.
The connectedness between the high-dimensional solution and the true solution ensures the possibility of the regularization effectively guiding the SDR method toward the true solution.
Hence, the utilization of effective regularization techniques is crucial in achieving favorable outcomes in this context.

\section{Comparison with the Neural Network}
\label{Comparison to the Neural Network}
In the context of the SNL problem, our findings demonstrate the consistent superiority of the high-dimensional SDR method over the low-dimensional approach. 
As previously discussed, the high-dimensional method outperforms the low-dimensional method in the SNL problem, which minimizes the loss function in $\mathbb{R}^d$.
Moreover, existing literature (e.g., \cite{biswas_five_people_2006, biswas_lian_wang_ye_2006}) suggests that utilizing the high-dimensional solution as a warm-start for the low-dimensional method yields even better performance than both the low-dimensional and high-dimensional methods in the SNL problem. 
This highlights the ability of high-dimensional models to enhance low-dimensional optimization by providing an advantageous warm-start.

Interestingly, similar phenomena can be observed in the field of neural networks, where researchers and practitioners have increasingly focused on developing larger and more complex architectures within the realm of deep learning. 
Notable examples include AlphaFold, AlphaGo, ResNet, GPT-3, Transformer, and VGG \cite{alphafold, alphago, resnet, gpt3, transformer, vgg}. 
This shift toward larger neural networks indicates their ability to achieve high performance across various tasks. 
Additionally, large neural networks can serve as effective initializations, providing a warm-start for new networks in different tasks. 
This process is commonly referred to as transfer learning.

In the following paragraphs, we will show the high performance of larger neural networks and their capability to facilitate transfer learning.

Larger neural network models exhibit superior performance primarily due to their capacity to generalize and their capability to extract diverse and expressive features. This characteristic results in improved performance and generalization on a given task, as well as on new, unseen data. The effectiveness of larger models has been demonstrated by neural networks such as VGG, ResNet, DenseNet, and large-scale Transformer models like BERT and GPT-3 \cite{vgg, resnet, densenet, bert, gpt3}.

Furthermore, larger models hold promise for facilitating transfer learning in various domains, as evidenced by studies such as \cite{tl_survey_2019}. Transfer learning involves fine-tuning or extracting features from a pre-trained model on a specific task or domain. Pre-trained transformers like BERT \cite{tl_by_bert, language_finetuning} and large CNNs such as ResNet \cite{tl_by_resnet} have achieved state-of-the-art performance on diverse natural language processing tasks with relatively small amounts of task-specific data.

In the following sections, we will first discuss the blessing of high-dimensional models, emphasizing their exceptional ability to extract more information from the problem and their generalization ability. 
Subsequently, we will explore the capability of high-dimensional models to serve as warm-starts for new models.

\subsection{Blessing of a High-Dimensional Model}

The blessing of high-dimensional models can be observed in other applications of SDP and designs of neural networks.
Augmenting the dimension of the problems can extract more information from the problems and make the process of optimization easier. 
Therefore, if the problem is difficult, increasing the dimension of the problem with a possible relaxation may reach an acceptable high-dimensional solution.
This leaves the possibility to project it onto a lower-dimensional space and get a solution with better accuracy, which was suggested by \cite{biswas_five_people_2006, biswas_lian_wang_ye_2006}.
Therefore, in neural networks and optimization, high-dimensional models have more generalization ability.

For instance, the semi-definite relaxation of the max-cut problem (see \cite{maxcut_gwalgorithm}) and the sensor network localization (SNL) problem both follow a paradigm of second-order dimension augmentation, including increasing the dimension of the problem, relaxing the problem, obtaining a warm-start from the high-dimensional solution and using it for better optimization. 
Increasing the dimension of the problem can lead to a better quality of the solution, which was suggested by \cite{biswas_five_people_2006, biswas_lian_wang_ye_2006}. 
Here is a flow chart of the SDR method in the SNL problem, where $z$ is the padded vector of $Z=\text{mat}(z)$, $\text{PSD}^{n+d}$ is the cone of padded positive semi-definite $(n+d)\times (n+d)$ matrices, and $A,g$ are calculated by the SNL problem. 
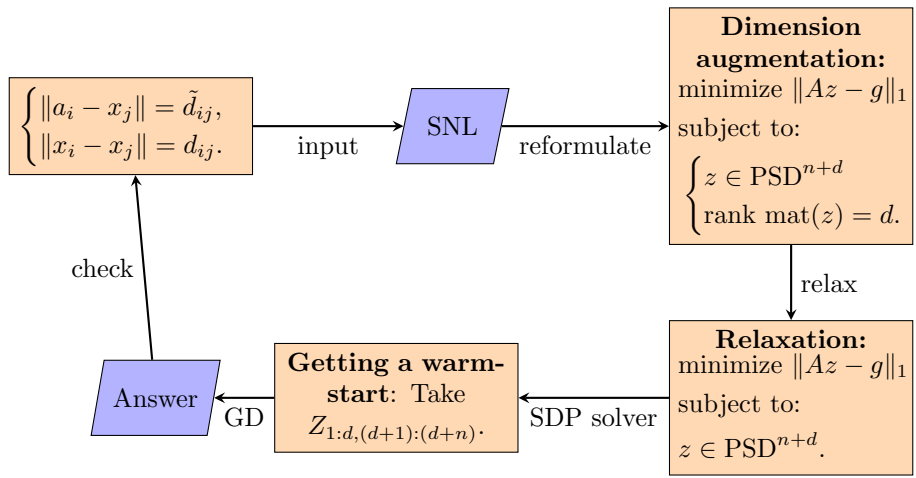
\begin{figure}[H]
\centering
\begin{tikzpicture}[node distance=2cm]

\node (prob) [process]{
$
\begin{cases}
\Vert a_i-x_j\Vert  = \tilde{d}_{ij},\\
\Vert x_i-x_j\Vert  = d_{ij}.
\end{cases}
$
};

\node (in1) [io, right =of prob]
{SNL};
\node (pro1) [process, right of=in1, xshift=2.5cm] 
{\textbf{Dimension augmentation:}
$
\begin{aligned}
&\text{minimize } \Vert Az-g\Vert_1\\
&\text{subject to:}\\
&\begin{cases}
    z \in \text{PSD}^{n+d} \\
    \text{rank mat} (z) = d.
\end{cases}
\end{aligned}
$
};
\node (pro2) [process, below =of pro1, yshift=1cm] 
{\textbf{Relaxation:}  
$
\begin{aligned}
&\text{minimize } \Vert Az-g\Vert_1\\
&\text{subject to:}\\
& z \in \text{PSD}^{n+d}. \\
\end{aligned}
$
};
\node (pro3) [process, left = of pro2]
{
\textbf{Getting a warm-start}:
Take $Z_{1:d, (d+1):(d+n)}$.
};

\node (out1) [io, left of=pro3, xshift=-1.2cm] {Answer};

\draw [arrow] (in1) -- node[anchor=north] {reformulate} (pro1);
\draw [arrow] (prob) -- node[anchor=north] {input} (in1);
\draw [arrow] (pro1) -- node[anchor=west] {relax} (pro2);
\draw [arrow] (pro2) -- node[anchor=north] {SDP solver} (pro3);
\draw [arrow] (pro3) -- node[anchor=north] {GD} (out1);
\draw [arrow] (out1) -- node[anchor=east] {check} (prob);

\end{tikzpicture}

\caption{A flowchart of the SDR algorithm in SNL}
\end{figure}

Similarly, a high-dimensional model of a neural network can extract more features from the dataset. 
With a well-designed structure, large neural networks are enabled to have more generalization ability than small neural networks.
For instance, Transformer \cite{transformer} and GPT \cite{gpt3} can capture long-term dependencies in the dataset, which improves the performance of many natural language processing tasks.
As for Computer Vision tasks, ResNet \cite{resnet} and VGG \cite{vgg}  are examples of high-dimensional models that can perform well on a wide range of image classification problems.

Generally, high-dimensional models learn more complex features than low-dimensional models and can outperform the latter under certain conditions.
However, such models may also potentially overfit the data as the number of parameters increases.
Therefore, adding a regularizer or performing transfer learning is necessary to avoid the overfitting issue.
In the scenario of solving SNL problems, the importance of regularization in optimization was also suggested by experiments in \cite{biswas_five_people_2006,biswas_lian_wang_ye_2006}.

Furthermore, a well-formed larger model may contain some specific structures of higher order.
For instance, the method of semi-definite relaxation of the SNL problem exploits the structure of the problem by iterating the relaxation of the vector product $Y_{ij}$ of the expected position of the sensors to minimize the loss function
$$
    \sum_{\{i,j\}\in N_x} |Y_{ii} - 2Y_{ij} + Y_{jj} - d_{ij}^2| + \sum_{(i,j) \in N_a} |a_{i}^T a_i - 2a_{i}^T x_j \ + Y_{jj} - \tilde{d}_{ij}^2|.
$$
Moreover, in the max-cut problem, one ultimately wants to solve
\begin{align}
\begin{split}
\label{maxcut origin}
    \text{maximize}\quad &
            \frac12 \sum_{1\leq i<j \leq M} r_{ij} (1-b_ib_j), \\
    \text{subject to} \quad &
            |b_i| = 1, \quad i = 1,\dots, M,
\end{split}
\end{align}
where $r_{ij}\geq 0$ are known. 
One may do a second-order augmentation as 
\begin{align}
\begin{split}
\label{maxcut relaxed}
    \text{maximize}\quad &
            \frac12 \sum_{1\leq i<j \leq M} r_{ij} (1-B_{ij}), \\
    \text{subject to} \quad &
            B_{ii} = 1, \quad i = 1,\dots, M, \\
            & B \succeq 0.
\end{split}
\end{align}
This process increases the dimension of the optimization problem we solve from $M$, the dimension of \eqref{maxcut origin}, to $O(M^2)$, the dimension of \eqref{maxcut relaxed}, which received theoretical and experimental success; see, e.g., \cite{maxcut_gwalgorithm, semidefiniterelax}. 

In comparison, the transformer architecture uses the scaled product to train attention over data, as described by the equation:
$$
    \text{Attention} (Q, K, V) = \text{softmax} \left( \frac{QK^T}{\sqrt{d_k}}\right) V.
$$
Here, the matrices $Q,K$ and $V$ represent the queries, keys and values \cite{transformer}, and $d_k$ is the dimension of the keys.
This matrix product is used similarly; it helps the model to identify which queries are important and then scales the attention to each one. 
The scaled product is used to ensure that the attention is evenly distributed across all of the queries.
This suggests that the matrix product can sometimes be a good structure to increase the dimension of the problem.

\subsection{Blessing of a Warm-Start from a High-Dimensional Model}
In addition to increasing the dimension of a problem, the way to get a warm-start from the high-dimensional solution also plays an important role in solving the problem.

For example, in the max-cut problem, the solution is generated by rounding the solution obtained from the relaxed problem to a feasible solution for the original problem, which can be done using various methods, such as the Goemans-Williamson algorithm \cite{maxcut_gwalgorithm}. 
We recall the max-cut problem \eqref{maxcut origin} and its relaxed form \eqref{maxcut relaxed}.
The key process is to let 
$$
    b_j=1, \quad \text{if } v_i^T v \geq 0.
$$
where $v$ is a vector uniformly distributed on the unit sphere $S_n$ and $v_i$ is the $i$th column of $F$. Here, $F$ is the incomplete Cholesky decomposition of $B$, i.e., $B=F^TF$.
The methods used in this key step determine how well the solution to the relaxed problem can be translated back to the solution set to the original problem.

Similarly, in the SNL problem, the reularized SDR+GD method that uses the SDR solution as a warm-start of the loss function method has better performance than the pure high-dimensional method and pure low-dimensional method, which is shown in \cite{biswas_five_people_2006, biswas_lian_wang_ye_2006}.

Also, in the field of transfer learning, there are two kinds of analogies to the process of generating warm-starts from the high-dimensional models, which are fine-tuning and feature extraction. 
We refer the readers to \cite{tl_survey_2010, tl_survey_2019} for details.
 
Fine-tuning is the process of taking a pre-trained model and adapting its weights to new tasks.
It usually involves some amount of supervised training and typically requires using the same type of neural network architecture as the pre-trained model.
The new model is then initialized with the weights of the pre-trained model and then retrained on the new dataset. 
Fine-tuning is useful when the dataset is relatively small, as it allows for better generalization than training from scratch.

Feature extraction is the process of extracting useful features from a pre-trained model for use in another task.
Feature extraction usually involves extracting the last few layers of a pre-trained model and using them as inputs to a new model trained on the new task.
Feature extraction usually requires using a different type of neural network architecture than the pre-trained model.
It is useful for tasks with large datasets or with plenty of data to use for feature extraction.

In both cases, the trained large neural network provides a warm-start for the final neural network which we get from various training methods.

\section{Conclusions and Future Work}
\label{conclu}

The blessing of high-order dimensionality in scientific computing stems from the ability to exploit the structure of a problem by increasing the dimension of the problem and getting warm-starts from the high-dimensional optimization.
This approach can lead to more efficient algorithms and better generalization. 
However, it is essential to carefully design the larger model, considering the problem's structure, such as the order of the dimension augmentation, and the choice of generating a warm-start for the original problem. 
By doing so, one can tackle complex problems in various fields, such as real-world optimization problems and machine learning, by increasing the dimension of the problem.

Taking SNL as an example, we have demonstrated in Section \ref{Non-Convexity of the Low-Dimensional Optimization} that the low-dimensional method of minimizing the loss function does not work well, as the loss function is often non-convex. 
Our main theorem \ref{main_theorem} clearly illustrates this fact in terms of probability.
In contrast, the high-dimensional SDR method introduced in Section \ref{Convexity of the High-Dimensional Optimization} skillfully reformulate the SNL problem as a convex problem. 
Additionally, we have proved in Section \ref{Non-Convexity of Direct Dimension Augmentation} that the method of direct dimension augmentation usually minimizes a non-convex objective function.
Furthermore, we give an explanation for why the SDR+GD method that uses the SDR solution as a warm-start of the loss function method works in Section
\ref{High-dimensional Solution as a Warm-Start of the Low-dimensional Optimization}. 

We emphasize in Section \ref{Comparison to the Neural Network} the connection between our findings and some similar methods in neural networks, as well as the blessing of dimensionality in the fields of optimization and machine learning. 

Future work related to this paper includes:
\begin{itemize}
\item Generalization of the product-type dimension augmentation.

As shown above, the product-type dimension augmentation technique is used to solve SNL and max-cut problems.
We wonder if this technique can be used for more optimization problems.
More specifically, we wonder if there is a more general type of non-convex optimization problem that can be solved by this technique.

\item Completely rigorous explanation for the success of the SDR+GD method.

We hope a completely rigorous explanation can be given to explain why the warm-start given by the SDR method works so well in the SDR+GD method.
Moreover, we wonder if there exists regularization in the SDR method that can help the SDR method solve all SNL problems generated by globally rigid frameworks.

\end{itemize}

\textbf{Acknowledgment.}
We are grateful to Jiajin Li, Chunlin Sun, Bo Jiang and Qi Deng for their valuable remarks and suggestions which helped us to improve the paper and the code.

\bibliographystyle{plain}
\bibliography{bib}

\end{document}